%% file: root.tex
\documentclass{IEEEtran}

\IEEEoverridecommandlockouts                              

\usepackage{graphics} 
\usepackage{amsmath} 
\usepackage{amssymb}  
\usepackage{amsthm}
\usepackage{xcolor}
\usepackage{bbold}
\usepackage{cite}
\usepackage{graphicx}
\usepackage{subcaption}

\usepackage{tikz}
\usepackage{amsmath}
\usepackage{amsfonts}
\usepackage{amsmath,amssymb,amsfonts,color}
\usetikzlibrary{shapes,arrows}
\usetikzlibrary{arrows,calc}
\tikzset{
    block/.style = {draw, rectangle, 
        minimum height=1cm, 
        minimum width=2cm},
    input/.style = {coordinate,node distance=1cm},
    output/.style = {coordinate,node distance=3cm},
    arrow/.style={draw, -latex,node distance=2cm},
    pinstyle/.style = {pin edge={latex-, black,node distance=2cm}},
    sum/.style = {coordinate, node distance=1cm}
}

\input{header.tex}

\title{\LARGE \bf
On Centralized and Distributed Mirror Descent: Convergence Analysis Using Quadratic Constraints
}

\author{Youbang Sun$^{1}$, Mahyar Fazlyab$^{2}$ and Shahin Shahrampour$^{1}$
\thanks{ This work is supported in part by NSF ECCS-2136206 Award. }
\thanks{$^{1}$ Y. Sun and S. Shahrampour are with the Department of Mechanical and Industrial Engineering at Northeastern University, Boston, MA 02115, USA. 
        {\tt\small email:\{sun.youb,s.shahrampour\}@northeastern.edu}.}
\thanks{$^{2}$ M. Fazlyab is with the Department of Electrical and Computer Engineering at Johns Hopkins University, Baltimore, MD 21218, USA. {\tt \small email: mahyarfazlyab@jhu.edu}.}
}

\begin{document}

\maketitle
\thispagestyle{empty}
\pagestyle{empty}

\begin{abstract}
Mirror descent (MD) is a powerful first-order optimization technique that subsumes several optimization algorithms including gradient descent (GD). In this work, {\bl we develop a semi-definite programming (SDP) framework to analyze the convergence rate of MD in centralized and distributed settings under both strongly convex and non-strongly convex assumptions.} We view MD with a dynamical system lens and leverage quadratic constraints (QCs) to {\bl provide explicit convergence rates} based on Lyapunov stability. For centralized MD under strongly convex assumption, we develop a SDP that certifies exponential convergence rates. We prove that the SDP always has a feasible solution that recovers the optimal GD rate as a special case. {\bl We complement our analysis by providing the $O(1/k)$ convergence rate for convex problems.} Next, we analyze the convergence of distributed MD and characterize the rate using SDP. To the best of our knowledge, the numerical rate of distributed MD has not been previously reported in the literature. We further prove an $O(1/k)$ convergence rate for distributed MD in the convex setting. Our numerical experiments on strongly convex problems indicate that our framework certifies superior convergence rates compared to the existing rates for distributed GD.
\end{abstract}

\section{Introduction}
Over the last two decades, distributed optimization over multi-agent networks has received a lot of attention in control, optimization, machine learning, and signal processing. 
In distributed optimization,  a group of $n$ agents are connected via a graph and can communicate locally with their neighbors. Each agent is assigned a local objective function $f_i: \mathbb{R}^d \rightarrow \mathbb{R}$, and the agents aim to collectively minimize the global objective function, {\bl 
\begin{equation}\label{mainproblem}
    \min_{x \in \mathbb{R}^d } \left\{ f(x) \triangleq \frac{1}{n}\sum_{i=1}^n f_i(x)\right\}.
\end{equation}}
The most intuitive gradient-based algorithm to tackle the problem above is \textit{distributed gradient descent} \cite{nedic2009distributed}, where at each iteration $k$, each agent $i$ updates its decision variables by a (private) local gradient descent combined with an averaging of its neighbors variables. In the unconstrained case, this update is given by
$$x_i^{(k+1)} = x_i^{(k)} - \eta^{(k)} \nabla f_i(x_i^{(k)})+ \beta \sum_{j\in \mathcal{N}_i} (x_j^{(k)} - x_i^{(k)}),$$
where $\eta^{(k)}>0$ is the step-size and $\beta>0$ is the consensus parameter. In the form given above, this update is able to achieve optimal rates for convex problems using a diminishing step-size sequence. Optimality here refers to matching the centralized convergence rate (iteration complexity) up to some errors related to the network structure. However, when the local functions are smooth, the centralized GD algorithm employs a {\it constant} step-size sequence for which the above distributed counterpart fails to converge.


The mirror descent (MD) algorithm \cite{nemirovsky1983originalmd} is a primal-dual method that has been actively studied in recent years. MD can be seen as a generalization of GD, in which the squared Euclidean distance is replaced by Bregman divergence as the regularizer. The freedom in the choice of Bregman divergence makes MD suitable for various problem geometries. MD has been proven to have the same iteration complexity as GD for non-strongly convex problems \cite{beck2003mirror}, and it may even scale better with respect to the dimension of the decision variables \cite{ben2001ordered}. In the strongly convex scenario, MD is less studied, and very recently its exponential convergence was established under the Polyak-Łojasiewicz (PL) condition \cite{radhakrishnan2020linear}. Inspired by the success of MD in centralized optimization, MD has also been studied in the distributed setting. To the best of authors' knowledge, the convergence rate of distributed MD is not established for strongly convex and smooth problems, and only recently \cite{sun2020linear} provided a continuous-time analysis suggesting local exponential rate (without explicitly characterizing the rate).

In this paper, we leverage the framework of Quadratic Constraints (QCs) to certify numerical exponential convergence rates for centralized as well as distributed MD for strongly convex and smooth problems using SDP. {\bl For merely convex and smooth problems, we also establish an ergodic ${O}(1/k)$ convergence rate}.
%
We first analyze centralized MD, for which we derive linear matrix inequalities (LMIs) as sufficient conditions for convergence of the algorithm at a specified rate (Theorem \ref{centralized_theorem}, Theorem \ref{centralized_theorem_convex} and Proposition \ref{SDP_transform}). 
%
For the strongly convex case, we prove that these LMIs always have a feasible solution that matches the optimal convergence rate of GD when the Bregman divergence is chosen as the squared Euclidean distance (Proposition \ref{optrho} and Corollary \ref{gd_corollary}). Next, we analyze the convergence of distributed MD and characterize the rate using LMIs (Theorem \ref{distributed_theorem}, \ref{distributed_theorem_convex}). To the best of our knowledge, the exponential rate of distributed MD has not been previously established in the literature. 
Our numerical experiments on strongly convex problems indicate that our framework certifies superior convergence rates compared to the existing rates for distributed GD.

\subsection{Related Literature}
\subsubsection{Distributed {\bl Optimization}}\label{sec:dgd}
To ensure that distributed GD (or sub-gradient descent) reach consensus, many methods \cite{ nedic2009distributed,nedic2010constrained, lin2016distributed} use diminishing step-size (commonly $1/k$). For distributed MD, similar studies have been conducted for stochastic optimization \cite{yuan2018optimal,7383850} and online optimization \cite{shahrampour2017distributed,9070199}. Doan et al. \cite{8409957} provide convergence results for both centralized and decentralized MD algorithms. However, convergence rates obtained using diminishing step-size are sub-exponential and sub-optimal under assumptions of {\bl strong convexity} and smoothness. 

To address this issue, a number of recent works introduce an additional variable in the state vectors to track past gradients (see e.g., \cite{shi2015extra,qu2017harnessing, sun2019convergence, pu2020distributed}). One of the earlier works in this direction is the EXTRA algorithm proposed by Shi et al.  \cite{shi2015extra}, which uses the information from past two iterations to perform each update. For smooth problems, EXTRA provably achieves $O(1/k)$ convergence rate under the convexity assumption and exponential convergence rate under the strong convexity assumption, respectively. 

A closely relevant literature is the continuous-time distributed GD, where the algorithms are constructed by a set of ordinary differential equations (ODEs). These works are mostly based on the idea of \textit{integral feedback}, which can be thought as the continuous-time analog of gradient tracking. In this case, each agent uses an integration term as a part of the ODE (see e.g.,  \cite{cortes2013weightbalanceddigraph, 7744584, kia2015distributed, yang2016multi}). In these works, the analysis is carried out by proving the Lyapunov stability for the corresponding continuous-time dynamics, and exponential stability can be obtained in certain cases \cite{kia2015distributed}.
For MD, the continuous-time algorithm in \cite{sun2020distributed, sun2020linear} and the discrete-time algorithm in \cite{yu2020rlc} both adapt the integral feedback (or gradient tracking) method and propose algorithms that do not suffer from sub-optimal convergence rates. Specifically, Sun et al. \cite{sun2020linear} propose a continuous-time distributed MD that achieves a ``local'' exponential rate for strongly convex problems, and Yu et al. \cite{yu2020rlc} provide an $O(1/k)$ convergence rate under the convexity assumption in discrete time. Nevertheless, the exponential rate of (discrete-time) distributed MD for strongly convex and smooth problems remains an open problem, which we target in the current work.  

\subsubsection{Integral Quadratic Constraints}
Deriving convergence rates for iterative optimization algorithms in the worst-case is an integral part of algorithm design. However, this procedure is not principled, requires a case-by-case analysis, and might lead to conservative rates. To automate convergence analysis and derive sharp convergence rates, several past works have used Integral Quadratic Constraints (IQCs) and semidefinite programming in various settings \cite{lessard2016analysis,fazlyab2018analysis,hu2017dissipativity,taylor2017smooth,dhingra2018proximal,aybat2020robust,scherer2021convex}, pioneered by the work in \cite{lessard2016analysis}.  IQCs are a tool from robust control to analyze dynamical systems that contain components that are nonlinear, uncertain, or difficult to model \cite{megretski1997system}. The basic idea is to abstract these troublesome components by constraints on their input and output signals. This approach to algorithm analysis can also guide the search for parameter selection in algorithm design. \cite{sundararajan2017robust,sundararajan2020analysis} are of particular relevance to our work. They  both provide IQC-based analysis of distributed gradient-based algorithms in strongly convex settings. Compared to these works, our framework focuses on distributed mirror descent in both strongly convex and convex settings.  


\section{Preliminaries}
\subsection{Notations}
The identity matrix of dimension $n$ is denoted by $I_n$ and the $n$-dimensional vector with all entries 1 is represented by $\mathbf{1}_n$. 
 We denote the set of $n$-dimensional symmetric matrices by $\mathbb{S}^n$. The positive (negative) semi-definiteness of matrix $M$ is denoted as $M\succeq 0$ ($M\preceq 0$). We use $\otimes$ and $\norm{\cdot}$ to denote the Kronecker product and spectral norm, respectively. {We define the norm of vector $v$ with respect to a positive semi-definite matrix $M$ as $\|v\|_M$}. The indicator function of a set $\mathcal{X} \subseteq \mathbb{R}^d$ is defined as $\mathbb{I}_{\mathcal{X}}(x)=0$ if $x \in \mathcal{X}$ and $\mathbb{I}_{\mathcal{X}}(x)=+\infty$ otherwise. 

\begin{definition}[Strong convexity]
A differentiable function $f: \mathbb{R}^d \rightarrow \mathbb{R}$ is $\mu_f$-strongly convex on $\mathbb{R}^d$ if the following inequality is true for all $x, y \in \mathbb{R}^d$.
$$f(x) + \nabla f(x)^\top (y - x) + \frac{\mu_f}{2}\norm{y-x}^2 \leq f(y).$$
\end{definition}

\begin{definition}[Lipschitz smoothness]
A differentiable function $f: \mathbb{R}^d \rightarrow \mathbb{R}$ is $L_f$-smooth on $\mathbb{R}^d$ if $\frac{L_f}{2}\|x\|^2-f(x)$ is convex, which implies that for all $x, y \in \mathbb{R}^d$.
$$f(y) \leq f(x) + \nabla f(x)^\top (y - x) + \frac{L_f}{2}\norm{y-x}^2.$$
\end{definition}
We further denote the condition number of function $f$ by $\kappa_f \triangleq \frac{L_f}{\mu_f}\geq 1.$ When $\mu_f=0$, the function is only convex.

\begin{proposition}\label{IQC}
Suppose $f$ is $\mu_f$-strongly convex and $L_f$-smooth on $\mathbb{R}^d$. Then, the following inequality holds for all $x,y \in \mathbb{R}^d$, 
and $u=\nabla f(x), \ v = \nabla f(y)$,
\begin{equation} \label{eq: QC for SC f}
    \begin{aligned}
        &\begin{bmatrix}x - y\\ u - v \end{bmatrix}^\top
\begin{bmatrix}
\frac{-\mu_f L_f}{\mu_f + L_f}I_d & \frac{1}{2}I_d\\
\frac{1}{2}I_d & \frac{-1}{\mu_f + L_f}I_d
\end{bmatrix}   \begin{bmatrix}x - y\\u- v \end{bmatrix} \geq 0{\bl .}
    \end{aligned}
\end{equation}
\end{proposition}
The above QC follows from the combination of strong convexity and Lipschitz smoothness \cite{nesterov2018lectures,lessard2016analysis}.

\subsection{Centralized Mirror Descent Algorithm}
We start by providing some background on the centralized MD algorithm. For simplicity in the exposition, we study the unconstrained case, but our analysis can also be extended to the constrained case. Let us start with the GD algorithm, whose update is equivalent to the following minimization,
\begin{align*}
    x^{(k+1)} \!= &\underset{x}{\argmin} \bigg\{ f(x^{(k)}) +  \nabla f(x^{(k)})^\top (x \!- \!x^{(k)})\\ + &\frac{1}{2\eta} \|x \!-\!x^{(k)}\|^2\bigg\},
\end{align*}
where $\eta>0$ is the step-size.
In each iteration, the algorithm seeks to minimize a first-order approximation of the function with a Euclidean regularizer. As a generalization of gradient descent, MD replaces the {\bl squared} Euclidean distance with Bregman divergence, which is defined with respect to a distance generating function (DGF) $\phi:\R^d\to \R$ as follows
\begin{align}
    \Dc_{\phi} (x, x') \triangleq \phi(x) -\phi(x') - \langle \nabla\phi(x'), x-x'\rangle.
\end{align}

\begin{assumption}\label{Assumption_phi}
The distance generating function $\phi: \R^d \rightarrow \R$ is $\mu_\phi$-strongly convex and $L_\phi$-smooth.
\end{assumption}

The centralized (unconstrained) MD algorithm with step-size $\eta$ is written as 
\begin{equation}\label{originalmd}
    \begin{aligned}
        x^{(k+1)} \! &= \! \underset{x}{\argmin} \bigg\{ f(x^{(k)}) +  \nabla f(x^{(k)})^\top (x - x^{(k)})\\
        &\qquad+ \frac{1}{\eta}\Dc_\phi(x, x^{(k)})\bigg\},
    \end{aligned}
\end{equation}
where if we choose the Bregman divergence to be the squared Euclidean distance, the update above reduces to GD.

We can also view the MD update through a different lens using the {\it convex conjugate} of function $\phi$. 
The convex conjugate of function $\phi$, denoted by $\phi^\star$, is defined as
$\phi^\star(z) \triangleq \sup_{x \in \R^d} \{\langle x, z\rangle - \phi(x)\}$. Assumption \ref{Assumption_phi} guarantees that $\phi^\star$ is $L_\phi^{-1}$-strongly convex and $\mu_\phi^{-1}$-smooth. We refer the reader to \cite{hiriart2012fundamentals} for further details. Correspondingly, the following equivalence can be established,
$$
z' = \nabla\phi(x') \Longleftrightarrow x' = \nabla\phi^\star(z').
$$
Then, the centralized MD update can be rewritten in the following form,
\begin{equation} \label{step_by_step_md}
    \begin{aligned}
        z^{(k+1)} &= z^{(k)} - \eta\nabla f(x^{(k)})\\
        x^{(k+1)} &= \nabla\phi^\star(z^{(k+1)}),
    \end{aligned}
\end{equation}
or, equivalently, $ z^{(k+1)} = z^{(k)} - \eta (\nabla f \circ \nabla \phi^\star)(z^{(k)})$, 
%
which is reminiscent of GD. We can see that MD is more general than GD in that we can exploit the geometry of the problem using an appropriate choice of $\phi$, {\bl which makes MD more suitable for problems such as convex clustering, matrix optimization with regularization, etc. \cite{lashkari2007convex,benfenati2020proximal}.} 

Denoting $x^\star$ and $z^\star$ as the fixed points of \eqref{step_by_step_md}, we have 
\begin{align*}
    z^\star = z^\star - \eta \nabla f(x^\star) \quad x^\star = \nabla \phi^\star(z^\star),
\end{align*}
which implies that $x^\star$ is a minimizer of $f$.

\section{Convergence Analysis of Centralized Mirror Descent}\label{centralzied_section}
In this section, we provide a convergence analysis of the centralized MD using semidefinite programming. 
%
%
{\bl Our starting point is to describe all the nonlinear functions in the algorithm, namely $\nabla f$ and $\nabla \phi^\star$ by QCs on their input-output pairs, resulting in a  \emph{quadratically-constrained linear system}. We then find a suitable ``rate-generating'' Lyapunov function for this constrained system using semidefinite programming. We derive exponential (respectively, sub-exponential) convergence rate for strongly convex (respectively, convex) functions.} 


\subsection{Exponential Convergence for Strongly Convex $f$}\label{centralized_convergence}
%
In the following theorem, we characterize an LMI that depends on parameters of $f$ ($\mu_f$ and $L_f$), parameters of $\phi$ ($\mu_\phi$ and $L_\phi$), and several decision variables (including the step-size $\eta$ and the convergence rate $\rho \in (0,1)$). We prove that if the LMI is satisfied, the iterates converge exponentially fast to the unique fixed point $(x^\star,z^\star)$ with the rate $\rho$.

\begin{theorem}\label{centralized_theorem}
Let Assumption \ref{Assumption_phi} hold and assume that $f$ is $\mu_f$-strongly convex and $L_f$-smooth. Define matrices ${\bl M_{sc}}, M_f, M_\phi$ as follows, 
\begin{equation}\label{matrix_centralized}
    \begin{aligned}
        {\bl M_{sc}} &= \begin{bmatrix}\frac{1-\rho}{2\mu_\phi}I_d&0&0\\0&0&\frac{-\eta}{2} I_d\\0&\frac{-\eta}{2} I_d&\frac{\eta^2}{2\mu_{\phi}}I_d\end{bmatrix}   \\
        M_f &= \begin{bmatrix}
            0&0&0\\
            0&\frac{-\mu_f L_f}{\mu_f + L_f}I_d&\frac{1}{2}I_d\\
            0&\frac{1}{2}I_d&\frac{-1}{\mu_f + L_f}I_d
        \end{bmatrix} \\
        M_\phi &= \begin{bmatrix}
    \frac{-1}{\mu_\phi + L_\phi}I_d&\frac{1}{2}I_d&0\\
    \frac{1}{2}I_d&\frac{-\mu_\phi L_\phi}{\mu_\phi + L_\phi}I_d&0\\
    0&0&0
    \end{bmatrix}.  
    \end{aligned}
\end{equation}
If there exist some $\rho \in (0,1), \eta>0, \sigma_f \geq 0, \sigma_\phi \geq 0$, such that the following matrix inequality holds 
\begin{equation}
    \label{centralized_LMI}
    {\bl M_{sc}} + \sigma_f M_f + \sigma_\phi M_\phi \preceq 0,
\end{equation}
then the mirror descent algorithm in \eqref{step_by_step_md} converges exponentially fast with the rate of $\rho$. {\bl In particular,
$$
\|x^{(k)}-x^\star\|^2\leq \frac{2\Dc_{\phi^\star}(z^{(0)} , z^\star)}{\mu_\phi} \rho^k.
$$

}
\end{theorem}

\begin{proof}
Denote $u^{(k)} \triangleq  \nabla f(x^{(k)})$ and define the stacked vector

\begin{align}\label{eq:stack}
{e^{(k)}}^\top = \begin{bmatrix}(z^{(k)}-z^\star)^\top & (x^{(k)} - x^\star)^\top & (u^{(k)}- u^\star)^\top \end{bmatrix}.
\end{align}
Then, from Proposition \ref{IQC}, we obtain the following quadratic inequalities
\begin{align*}
    e^{(k)\top}M_f e^{(k)} \geq 0, \ \ e^{(k)\top}M_\phi e^{(k)} \geq 0 \quad \forall k,
\end{align*}
which are imposed by $\nabla f$ and $\nabla \phi$, respectively. Consider the Lyapunov candidate $V^{(k)} = \rho^{-k} \Dc_{\phi^\star}(z^{(k)} , z^\star).$ Recall that $\phi^\star$ is $L_\phi^{-1}$-strongly convex and $\mu_\phi^{-1}$-smooth, so the Lyapunov function is indeed non-negative and continuously differentiable. Using Lemma \ref{bregman_difference} (provided in the appendix of \cite{sun2021centralized}), we can calculate the Lyapunov function difference between two consecutive iterations as
\begin{equation}
    \begin{aligned}
        V^{(k+1)} - V^{(k)} \leq \rho^{-k-1} e^{(k)\top}{\bl M_{sc}} e^{(k)}.
    \end{aligned}
\end{equation}
%
%
Utilizing the two quadratic inequalities imposed by the nonlinearities, we can write
\begin{equation*}
    \begin{aligned}
        &\hphantom{=\ }V^{(k+1)} - V^{(k)}\leq \rho^{-k-1} e^{(k)\top}{\bl M_{sc}} e^{(k)}\\
        & \leq \rho^{-k-1} e^{(k)\top}({\bl M_{sc}}  + \sigma_f M_f  + \sigma_\phi M_\phi )e^{(k)}.
    \end{aligned}
\end{equation*}
Now if the LMI in \eqref{centralized_LMI} holds, the Lyapunov function is non-increasing, which yields
\begin{equation}
    \Dc_{\phi^\star}(z^{(k)} , z^\star) = \rho^k V^{(k)} \leq \rho^k V^{(0)} = \rho^k \Dc_{\phi^\star}(z^{(0)} , z^\star).
\end{equation}
    Observing $\Dc_{\phi^\star}(z^{(k)} , z^\star)=\Dc_{\phi}(x^\star,x^{(k)})$ and  
    $$
    \frac{\mu_\phi}{2}\|x^{(k)}-x^\star\|^2 \leq \Dc_{\phi}(x^\star,x^{(k)}),
    $$
    completes the proof.
\end{proof}

Theorem \ref{centralized_theorem} provides a matrix inequality feasibility problem that establishes the exponential convergence rate of MD for a given $\rho$. This matrix inequality is linear in $(\rho,\sigma_f,\sigma_{\phi})$ (but not in $\eta$), allowing us to find the smallest $\rho$ by the semidefinite program 
\begin{alignat}{2} \label{centralized_optimization}
& \underset{\rho,\sigma_{\phi}, \sigma_{f}}{\text{minimize}}\qquad
& & \rho  \\
& \text{subject to}
& & 0<\rho \leq 1 \notag \\
& & & \eta,\sigma_{\phi},  \sigma_{f}  \geq 0 \notag \\
& & & {\bl M_{sc}} + \sigma_f M_f + \sigma_\phi M_\phi \preceq 0. \notag
\end{alignat}
If in addition we want to optimize $\rho$ over the step-size $\eta$, we can use Schur Complements to ``convexify'' the matrix inequality with respect to $\eta$. We state this result formally in the next proposition.  

\begin{proposition}\label{SDP_transform}
The optimization problem in \eqref{centralized_optimization} is equivalent to the following SDP,

\begin{equation}\label{centralized_SDP}
\begin{aligned}
& \underset{\eta, \rho, \sigma_{\phi}, \sigma_{f}}{\text{minimize}}\qquad
& \rho  \\
& \text{subject to}
& 0<\rho \leq 1\\
& & \eta, \sigma_{\phi},\sigma_{f}  \geq 0\\
\end{aligned}
\end{equation}
\begin{equation*}\resizebox{0.99\hsize}{!}{$    \begin{bmatrix}
     \frac{\sigma_\phi}{\mu_\phi + L_\phi} + \frac{\rho-1}{2\mu_\phi}&\frac{- \sigma_\phi }{2}&0&0\\
     \frac{- \sigma_\phi }{2}&\frac{\mu_\phi L_\phi\sigma_\phi}{\mu_\phi + L_\phi} +\frac{\mu_\phi }{2} + \frac{\mu_f L_f\sigma_f}{\mu_f + L_f}&\frac{- \sigma_f }{2}&\frac{-\sqrt{\mu_\phi}}{\sqrt{2}}\\
     0&\frac{- \sigma_f }{2}&\frac{\sigma_f}{\mu_f + L_f}&\frac{\eta  }{\sqrt{2 \mu_\phi}}\\
     0&\frac{-\sqrt{\mu_\phi}}{\sqrt{2}}&\frac{\eta  }{\sqrt{2 \mu_\phi}}&1
     \end{bmatrix} \succeq 0 $}
\end{equation*}
\end{proposition}
 We refer to the appendix of \cite{sun2021centralized} for the proof of this proposition. We now show that the SDP in \eqref{centralized_SDP} has a feasible solution for which we can analytically calculate the convergence rate. 
\begin{proposition}\label{optrho}
The following selection
\begin{align}
\label{centralized_rate}
   \eta &= \sigma_f = \frac{2 \mu_\phi}{\mu_f + L_f} \\ \notag 
    \sigma_\phi &= \frac{4 \mu_f L_f}{(\mu_f + L_f)^2} \frac{(1+\kappa_\phi)}{\kappa_\phi(\kappa_\phi-1)}  \\ \notag
    \rho_{opt} &= 1- \frac{4 \mu_f L_f}{(\mu_f + L_f)^2 \kappa_\phi^2},
\end{align}
is a feasible solution to the SDP in \eqref{centralized_SDP}.
\end{proposition}
The proof of the proposition can be found in the appendix of \cite{sun2021centralized}. Note that $\rho_{opt}$ is an upper bound on the optimal value of \eqref{centralized_SDP}.

{\bl 
The recent work of \cite{radhakrishnan2020linear} also proposed an explicit rate of $1-\frac{1}{5\kappa_\phi^2\kappa_f^2}$~for MD under the PL condition. Though PL condition is weaker than strong convexity, $\rho_{opt}$ is strictly smaller than the rate of \cite{radhakrishnan2020linear}. Furthermore, in our result we do not make full use of strong convexity: we only require the quadratic inequality \eqref{eq: QC for SC f} to hold for the pair $(x,x^\star)$ ($x$ arbitrary and $x^\star$ the fixed point of the algorithm), whereas for strongly convex $f$ this inequality holds for \emph{all} $(x,y)$.
} Our rate also recovers the optimal rate of GD as a special case.
\begin{corollary}\label{gd_corollary}
For $\phi(x)=\frac{1}{2}\norm{x}^2$ the optimal rate $\rho_{opt}$ in \eqref{centralized_rate} coincides with the optimal convergence rate of gradient descent.
\end{corollary}
\begin{proof}
If $\phi(x)=\frac{1}{2}\norm{x}^2$, we have that $\phi^\star(z)=\frac{1}{2}\norm{z}^2$ and \eqref{step_by_step_md} is equivalent to GD. In this case, the condition number $\kappa_\phi = \frac{L_\phi}{\mu_\phi} = 1$, and $\rho_{opt}$ reduces to the optimal convergence rate for GD (see Theorem 2.1.15 in \cite{nesterov2018lectures}).
\end{proof}

{\bl
\subsection{$O(1/k)$ Convergence for Convex $f$}
We now propose an LMI which estbalishes subexponential convergence rate for the MD algorithm when the objective function is convex ($\mu_f=0$).

\begin{theorem}\label{centralized_theorem_convex}
Let Assumption \ref{Assumption_phi} hold and assume that $f$ is convex ($\mu_f=0$) and $L_f$-smooth ($0<L_f<\infty$), and define the matrix $M_c$ as follows, 
\begin{equation}\label{matrix_centralized_convex}
    \begin{aligned}
        M_c &= \begin{bmatrix}0&0&0\\0&0&\frac{\epsilon-\eta}{2} I\\0&\frac{\epsilon-\eta}{2} I&\frac{\eta^2}{2\mu_{\phi}}I\end{bmatrix}. 
    \end{aligned}
\end{equation}
If there exist some $\eta>0, \sigma_f \geq 0, \sigma_\phi \geq 0, \epsilon \geq 0$, such that the following matrix inequality holds 
\begin{equation}
    \label{centralized_LMI_convex}
    M_c + \sigma_f M_f + \sigma_\phi M_\phi \preceq 0,
\end{equation}
then the ergodic mean of function value at iteration $K$ satisfies
$$
{ f(\Bar{x}^{(K)})} - f(x^\star) \leq \frac{ \Dc_{\phi^\star}(z^{(0)},z^\star)}{\epsilon K},
$$
where $\Bar{x}^{(K)} = \frac{1}{K} \sum\limits_{i = 1}^K  x^{(i)}$.
\end{theorem}


}

{\bl 
We remark that a similar analysis can be applied to Theorem \ref{centralized_theorem_convex} to find the best step-size that maximizes $\epsilon$. The details are omitted due to space limitation.}

\begin{remark}[Constrained Mirror Descent]
Consider the constrained version of centralized (lazy) MD \cite{hazan2016introduction},
\begin{align}
\label{eq::constrained CMD 0}
\begin{aligned}
z^{(k+1)} &= z^{(k)}  -\eta \nabla f(x^{(k)}) \\
s^{(k)} &= \nabla \phi^\star(z^{(k)}) \\ 
x^{(k)} &= \arg\min_{x\in \mathcal{X}} \Dc_{\phi}(x,s^{(k)}),
\end{aligned}
\end{align}
where $\mathcal{X}$ is a convex subset of $\mathbb{R}^d$. By defining $g(x) = \mathbb{I}_{\mathcal{X}}(x)$ as the indicator function of the set $\mathcal{X}$ and denoting its subdifferential by $\partial g$, the optimality condition that characterizes $x^{(k)}$ is
\begin{align*}
\nabla \phi(x^{(k)}) - z^k & \in \partial g(x^{(k)}),
\end{align*}
%
Using the fact that the subdifferential $\partial g$ is monotone (since $\mathcal{X}$ is convex), we can rewrite \eqref{eq::constrained CMD 0} as
\begin{align}
\label{eq::constrained CMD 2}
z^{(k+1)} &= z^{(k)}  -\eta u^{(k)} \\
u^{(k)} & \triangleq \nabla f(x^{(k)}) \notag \\
v^{(k)} & \triangleq \nabla \phi(x^k), \notag
\end{align}
subject to the quadratic constraint
\begin{align*}
    (v^{(k)}-v^\star-(z^{(k)}-z^{\star}))^\top (x^{(k)}\!-\!x^{\star}) \geq 0 \quad \forall k,
\end{align*}
Furthermore, we can write two separate quadratic constraints for the relationships $u^{(k)}=\nabla f(x^{(k)})$ and $v^{(k)}=\nabla \phi(x^{(k)})$. We can therefore employ the same approach and derive an LMI as a sufficient condition to establish exponential and $O(1/k)$ convergence rates for strongly convex and convex problems, respectively.
\end{remark}
\section{Convergence Analysis of Distributed Mirror Descent}
%
In the distributed setup, we have a network of agents, characterized by an undirected graph $\Gc = (\Vc,\Ec)$, where each node in $\Vc = \{1,\ldots,n\}$ represents an agent, and the connection between two agents $i$ and $j$ is captured by the edge $\{i,j\} \in \Ec$. We use $\Nc_i\triangleq\{j \in \Vc: \{i,j\}\in \Ec\}$ to denote the neighborhood of agent $i$. The graph Laplacian is represented by $\Lc \in \R^{n\times n}$.
\begin{assumption}\label{network_assump}
The graph $\Gc$ is undirected and connected, i.e., there exists a path between any two distinct agents $i,j \in \Vc$.
\end{assumption}
The connectivity assumption implies that $\Lc$ has a unique null eigenvalue; that is, $\Lc \mathbf{1}_n=0$.






{\bl 
\subsection{Distributed Mirror Descent Algorithm}

We first introduce the distributed MD update, in which each agent $i$ in the network implements the following iterative algorithm,
\begin{equation}\resizebox{1\hsize}{!}{$\label{discrete_distributed}
    \begin{aligned}
        {z_i}^{(k+1)}&=  {z_i}^{(k)} \!- \eta_1 \big( \nabla f_i({x_i}^{(k)}) \!+\! {y_i}^{(k)}\big)\vphantom{\sum_{j \in \Nc_i}}   - \eta_2 \sum_{j \in \Nc_i} ({z_i}^{(k)} \!-\! {z_j}^{(k)})  ,
     \\
   {y_i}^{(k+1)} &= {y_i}^{(k)} \! +\! \eta_2\sum_{j \in \Nc_i} ({z_i}^{(k)} \!-\! {z_j}^{(k)})  \vphantom{\bigg)},\\
   {x_i}^{(k+1)} &= \nabla\phi^\star({z_i}^{(k+1)}) \vphantom{\sum_{j \in \Nc_i}\bigg)}.
\end{aligned}$}
\end{equation}
The first update uses private
gradient information as well as the dual variables from the neighbors. It also depends on a variable ${y_i}^{(k)}$ which acts as an integrator. This algorithm is similar to the discretized version of the distributed MD proposed in \cite{sun2020distributed} using the idea of integral feedback. However, the method differs slightly in the local averaging in that the algorithm in \cite{sun2020distributed} performs local averaging with respect to the primal variable, and here the averaging is done on the dual variable ${z_i}^{(k)}$. 

{\bl 
It is evident that the behavior of this system relies} on the network structure through the dependence on the Laplacian of the graph capturing the network. Since $\Lc \in \mathbb{S}^{n}$, the LMIs will consist of matrices whose dimensions scale with $n$, which is not suitable when $n$ is large. Following the idea in \cite{sundararajan2017robust,sundararajan2020analysis}, we transform the updates such that the dependence on the {\it full structure} of the network is avoided. Define 
$$W\triangleq I_n- \eta_2 \Lc = \Delta W+\frac{1}{n}\mathbf{1}_n\mathbf{1}_n^\top,$$
and further denote the spectral norm of $\Delta W$ by $ \lambda \triangleq \|\Delta W\|$. The quantity $1 - \lambda$ is also known as the spectral gap.
%

To represent the updates collectively for all the agents, we define the stacked vectors 
\begin{align} \label{dmd inputs}
    z^{(k)} &= [z_1^{(k)\top},\ldots, z_n^{(k)\top}]^\top \\ \notag
    y^{(k)} &= [y_1^{(k)\top},\ldots, y_n^{(k)\top}]^\top \\ \notag
    u^{(k)} &= \nabla \fb(x^{(k)})\triangleq[\nabla f_1(x_1^{(k)})^\top,\ldots,\nabla f_n(x_n^{(k)})^\top]^\top \\ \notag
    x^{(k)} &= [\nabla \phi^\star(z_1^{(k)})^\top,\ldots, \nabla\phi^\star(z_n^{(k)})^\top]^\top \\ \notag
    v^{(k)}&= (\Delta W \otimes I_d) z^{(k)}.
\end{align}
We can now rewrite \eqref{discrete_distributed} as 
\begin{equation}\label{DMD_spectral}
    \begin{aligned}
    z^{(k+1)} &=  (\frac{1}{n}\mathbf{1}_n\mathbf{1}_n^\top \otimes I_d)z^{(k)} - \eta_1 (u^{(k)}+y^{(k)})  + v^{(k)}   \\
    y^{(k+1)} &= y^{(k)} + ((I_n-\frac{1}{n}\mathbf{1}_n\mathbf{1}_n^\top) \otimes I_d)z^{(k)}-  v^{(k)} \\
    v^{(k)} &= (\Delta W \otimes I_d) z^{(k)}\\
    x^{(k)} &= \nabla \phi^\star(z^{(k)})\\
    u^{(k)} &= \nabla \fb(x^{(k)}).
    \end{aligned}
\end{equation}
To represent \eqref{DMD_spectral} in state-space form, we can write
\begin{equation} 
    \begin{aligned}
        \label{A_B}
    \begin{bmatrix}
    z^{(k+1)} \\ y^{(k+1)}
    \end{bmatrix} = &\begin{bmatrix}
    \frac{1}{n}\mathbf{1}_n\mathbf{1}_n^\top \otimes I_d & -\eta_1 I_{nd} \\  (I_n- \frac{1}{n} \mathbf{1}_n\mathbf{1}_n^\top)\otimes I_d & I_{nd}
    \end{bmatrix} \begin{bmatrix}
    z^{(k)} \\ y^{(k)}
    \end{bmatrix}   \\
    &+\begin{bmatrix}
    0 & -\eta_1 I_{nd} & I_{nd} \\ 0 & 0 & -I_{nd}
    \end{bmatrix} \begin{bmatrix}
      x^{(k)} \\ u^{(k)} \\ v^{(k)} 
    \end{bmatrix}.
    \end{aligned}
\end{equation}
%
Additionally, we know the following constraints on the updates,
\begin{equation}\label{C_D}
\resizebox{0.99\hsize}{!}{$
    \begin{bmatrix}
    0 \\ 0
    \end{bmatrix} = \begin{bmatrix}
    0 & \mathbf{1}_n\mathbf{1}_n^\top \otimes I_d\\  0 & 0
    \end{bmatrix} \begin{bmatrix}
    z^{(k)} \\ y^{(k)}
    \end{bmatrix}  
    + \begin{bmatrix}
    0 & 0 & 0 \\ 0 & 0 & \mathbf{1}_n\mathbf{1}_n^\top\otimes I_d
    \end{bmatrix} \begin{bmatrix}
      x^{(k)} \\ u^{(k)} \\ v^{(k)} 
    \end{bmatrix}$}.
\end{equation}
We define the state vector $\xi^{(k)\top} \triangleq \begin{bmatrix}
z^{(k)\top} & y^{(k)\top}
\end{bmatrix}$ as well as the input vector $\zeta^{(k)\top} \triangleq \begin{bmatrix}x^{(k)\top} & u^{(k)\top} & v^{(k)\top}\end{bmatrix}$. We can rewrite \eqref{A_B} and \eqref{C_D} as
\begin{equation}\label{new}
    \begin{aligned}
    \xi^{(k+1)} = A \xi^{(k)} + B \zeta^{(k)}~~~~~
    0 = F\xi^{(k)} + G \zeta^{(k)},
    \end{aligned}
\end{equation}
where $A,B,F,G$ are of appropriate dimensions. {\bl For the ease of notation we denote $H \triangleq \begin{bmatrix}F&G\end{bmatrix}$.

{\bl For the purpose of convergence analysis, we characterize the fixed point of \eqref{A_B}. Define $x^\star \triangleq \mathbf{1}_n  \otimes x_\star,$ where $x_\star \in \R^d$ is a minimizer of \eqref{mainproblem}, and let $z^\star \triangleq \nabla \phi(x^\star)$, $u^\star \triangleq \nabla \mathbf{f}(x^\star)$, $y^\star \triangleq -\nabla \mathbf{f}(x^\star)$ and $v^\star = 0$. By letting $z^{(k)}, y^{(k)}, v^{(k)}, x^{(k)}, u^{(k)}$ in \eqref{A_B} take the values of $z^\star, y^\star, v^\star, x^\star, u^\star$, it is easy to show that $z^{(k+1)} = z^{(k)}, y^{(k+1)} = y^{(k)}$ using Assumption \ref{network_assump}.
} 
} 
}
\subsection{Exponential Convergence of Distributed Mirror Descent}

In the following theorem, we present the main result of this section. We provide two LMIs to characterize the convergence rate of distributed MD. The LMIs are written in terms of several decision variables, including the step-size $\eta_1$ and the convergence rate $\rho$. If we can find a feasible solution for these LMIs, the distributed MD is guaranteed to converge exponentially fast.

Before stating the theorem, we state the following lemma, which will allow us to simplify the resulting SDP.
\begin{lemma}[Lemma 6 in \cite{sundararajan2017robust} ]\label{lemma_j}
 Suppose that square matrices $J_1, J_2$ satisfy $J_1^2 = J_1, J_2^2 = J_2, J_1J_2 = J_2 J_1 = 0.$ For square matrices $Q_1$ and $Q_2$, define $Q \triangleq Q_1 \otimes J_1 + Q_2\otimes J_2$. Then, the following are equivalent: 1) $Q \succeq 0.$ 2) $Q_1 \succeq 0, Q_2 \succeq 0.$
\end{lemma}

\begin{theorem}\label{distributed_theorem}
Let Assumptions {\bl \ref{Assumption_phi} and \ref{network_assump} hold and assume all local functions $f_i$ are $\mu_f$-strongly convex and $L_f$-smooth.} Define the following matrices,
\begin{align*}
    &A_1 = \begin{bmatrix}
    0 & -\eta_1  \\  1 & 1
    \end{bmatrix}, B_1 = \begin{bmatrix}
    0 & -\eta_1  & 1 \\ 0 & 0 & -1
    \end{bmatrix},\\
    &A_2 = \begin{bmatrix}
    1 & -\eta_1  \\  0& 1
    \end{bmatrix}, 
    B_2 = \begin{bmatrix}
    0 & -\eta_1  & 1 \\ 0 & 0 & -1
    \end{bmatrix},\\
     &  {\bl H_1=\begin{bmatrix}
    0 & 0 &0 & 0  & 0 \\  0 & 0&0 & 0  & 0
    \end{bmatrix},
    H_2 =\begin{bmatrix}
    0 & 1 &0 & 0  & 0 \\  0 & 0&0 & 0  & 1
    \end{bmatrix}.}
\end{align*}
Furthermore, define 
\begin{align*}
    M_f &=  \begin{bmatrix}
    0&0&0&0&0\\
    0&0&0&0&0\\
    0&0&\frac{-\mu_f L_f}{\mu_f + L_f}&\frac{1}{2}&0\\
    0&0&\frac{1}{2}&\frac{-1}{\mu_f + L_f}&0\\
    0&0&0&0&0
    \end{bmatrix}   
    \end{align*} \begin{align*} M_{\lambda} &= \begin{bmatrix}
     \lambda^2&0&0&0&0\\
    0&0&0&0&0\\
    0&0&0&0&0\\
    0&0&0&0&0\\
    0&0&0&0&-1\\
    \end{bmatrix} \\
    M_{\phi} & = \begin{bmatrix}
    \frac{-1}{\mu_\phi + L_\phi}&0&\frac{1}{2}&0&0\\
    0&0&0&0&0\\
    \frac{1}{2}&0&\frac{-\mu_\phi L_\phi}{\mu_\phi + L_\phi}&0&0\\
    0&0&0&0&0\\
    0&0&0&0&0
    \end{bmatrix}. 
\end{align*}
If there exists $\rho \in (0,1), \eta_1 \geq 0, P \in \mathbb{S}^{2}$, $P \succ 0, {\bl \Sigma_{eq}\in \mathbb{S}^{2} , \sigma_f \geq 0, \sigma_{\phi} \geq 0, \sigma_{\lambda} \geq 0}$, such that the following matrix inequalities hold for $i = 1, 2$
{\bl
\begin{equation}
    \begin{aligned}
        \label{distributed_LMI}
    &\begin{bmatrix}
     A_i^\top P A_i \!-\! \rho  P & A_i^\top P B_i \\ B_i^\top  P A_i & B_i^\top  P B_i
    \end{bmatrix} + \sigma_{f} M_{f} + \sigma_{\lambda} M_{\lambda}  \\
    &+ \sigma_{\phi}  M_{\phi} + H_i^\top \Sigma_{eq} H_i
    \preceq 0,
    \end{aligned}
\end{equation}
}
then the distributed MD algorithm \eqref{discrete_distributed} initialized at $y^{(0)}=0$ converges exponentially with a rate of $\rho$ as follows
$$\|\xi^{(k)}-\xi^\star\|^2_{P\otimes I_{nd}} \leq \rho^{k}  \|\xi^{(0)}-\xi^\star\|^2_{P\otimes I_{nd}}.$$
\end{theorem}
\begin{proof}
{\bl 
Define the vector $e^{(k)\top}=[ \xi^{(k)\top}  \quad \zeta^{(k)\top}]$. 
%
%
We can establish the following (in)equalities,
\begin{align*}
     e^{(k)\top} (M_f\otimes I_{nd}) e^{(k)} &\geq 0,\\
     e^{(k)\top} (M_\phi\otimes I_{nd}) e^{(k)} &\geq 0,\\
     e^{(k)\top} (M_\lambda\otimes I_{nd}) e^{(k)} &\geq 0,\\
     e^{(k)\top} H^\top (\Sigma_{eq}\otimes I_{nd}) H e^{(k)}& = 0.
\end{align*}
The first two inequalities are derived from Proposition \ref{IQC}, the third inequality is due to the fact that $\lambda = \norm{\Delta W}$, and the equality follows from the affine constraint in \eqref{C_D}.
   
   
}

Define the Lyapunov function
\begin{align*}
    V^{(k)} = \rho^{-k}(\xi^{(k)}-\xi^\star)^\top P' (\xi^{(k)}-\xi^\star),
\end{align*}
where $P' \triangleq P  \otimes I_{nd}$. Then, using \eqref{new} we can write
\begin{align*}
    V^{(k+1)} - V^{(k)} = \rho^{-k-1} e^{(k)\top}  \begin{bmatrix}
     A^\top P' A - \rho P'  & A^\top P' B \\ B^\top P' A & B^\top P' B
    \end{bmatrix}e^{(k)}.
\end{align*}
Now, if the following LMI holds
{\bl
\begin{equation}
    \begin{aligned}
        \label{distributed_LMI_2}
    &\begin{bmatrix}
     A^\top P' A - \rho  P' & A^\top P' B \\ B^\top  P' A & B^\top  P' B
    \end{bmatrix} + \sigma_{f} M_f\otimes I_{nd} \\
    &+ \sigma_{\lambda} M_{\lambda}\otimes I_{nd} + \sigma_{\phi}  M_{\phi}\otimes I_{nd} + H^\top (\Sigma_{eq}\otimes I_{nd}) H
    \preceq 0,
    \end{aligned}
\end{equation}

}
then for any $e^{(k)}$, we have that
\begin{equation*}
    \rho^{-k-1} e^{(k)\top}  \begin{bmatrix}
     A^\top P' A - \rho P' & A^\top P' B \\ B^\top P' A & B^\top P' B
    \end{bmatrix}e^{(k)} \leq 0,
\end{equation*}
or, equivalently, 
\begin{equation*}
    (\xi^{(k)}-\xi^\star)^\top P' (\xi^{(k)}-\xi^\star)\leq \rho^{k}  (\xi^{(0)}-\xi^\star)^\top P' (\xi^{(0)}-\xi^\star).
\end{equation*}
In words,  the squared norm of system variables decreases exponentially fast to zero. 
    
Next, we simplify the LMI such that the dimension is not dependent on the agent number $n$. Our approach follows that of \cite{sundararajan2017robust}. Define $J_1, J_2$ in Lemma \ref{lemma_j} as $J_1 = (I_n - \frac{1}{n}\mathbf{1}_n\mathbf{1}_n^\top)\otimes  I_d, J_2 =\frac{1}{n}\mathbf{1}_n\mathbf{1}_n^\top\otimes  I_d $. It is easy to verify that these matrices satisfy the constraints in Lemma \ref{lemma_j}. We then have that
\begin{equation*}
    \begin{aligned}
    A &= A_1 \otimes J_1 + A_2 \otimes J_2,\\
    B &= B_1 \otimes J_1 + B_2 \otimes J_2,\\
     H &= {\bl H_1 \otimes J_1 + H_2 \otimes J_2.}\\
    \end{aligned}
\end{equation*}
Since matrices $J_1, J_2$ satisfy the conditions in Lemma \ref{lemma_j}, if we consider $Q, Q_1, Q_2$ as the negative left hand side of \eqref{distributed_LMI_2}, \eqref{distributed_LMI} respectively, then a feasible set of solutions that satisfy  \eqref{distributed_LMI} is equivalently a feasible set of solutions for \eqref{distributed_LMI_2}, which completes our proof.
\end{proof}
The theorem provides two LMIs that establish the exponential convergence rate of distributed MD. As we can see the LMIs are more involved compared to the centralized case, and it is challenging to find even a suboptimal analytical rate.   

\begin{figure*}[t!]
\begin{subfigure}{0.3\linewidth}
    \includegraphics[width=\linewidth]{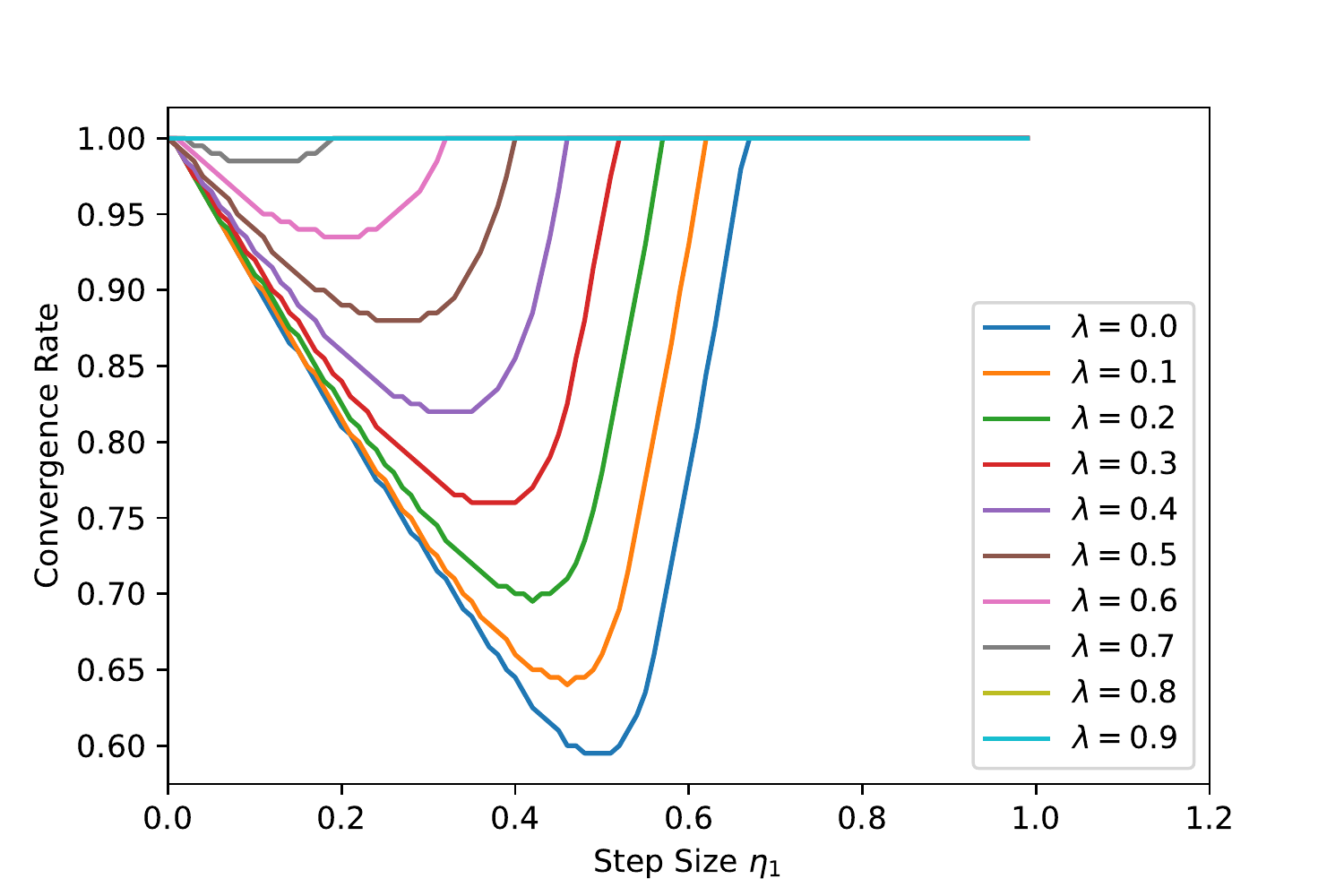}
    \caption{Convergence rates generated from Theorem \ref{distributed_theorem} versus step-size $\eta_1$.}
\label{lambda_dif}
\end{subfigure}
\hfill
\begin{subfigure}{0.3\linewidth}
    \includegraphics[width=\linewidth]{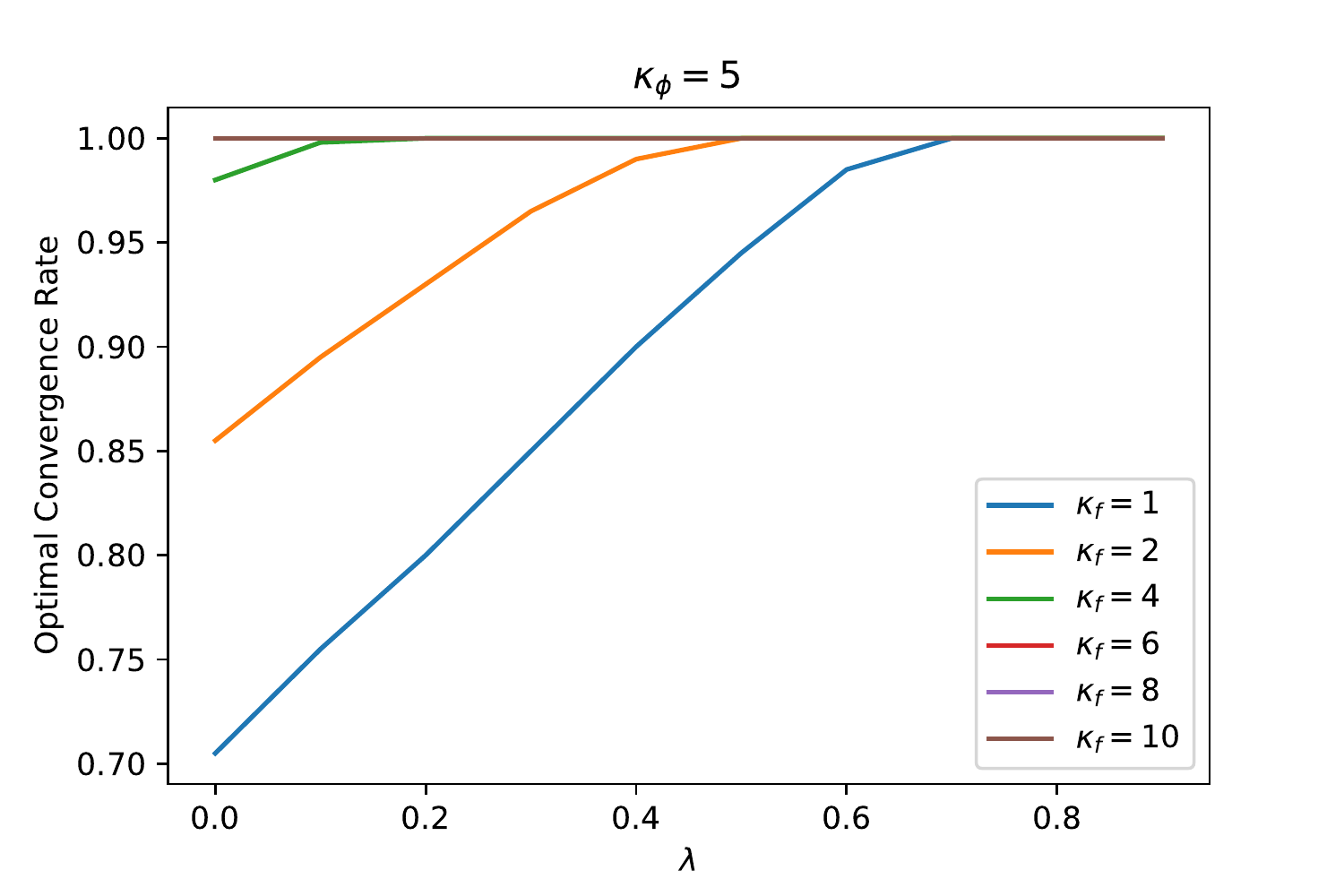}
    \caption{Fixed $\kappa_\phi = 5$. Varying $\lambda$ and $\kappa_f$, optimal learning rate is chosen.}
    \label{fig:1b}
\end{subfigure}
\hfill
\begin{subfigure}{0.3\linewidth}
    \includegraphics[width=\linewidth]{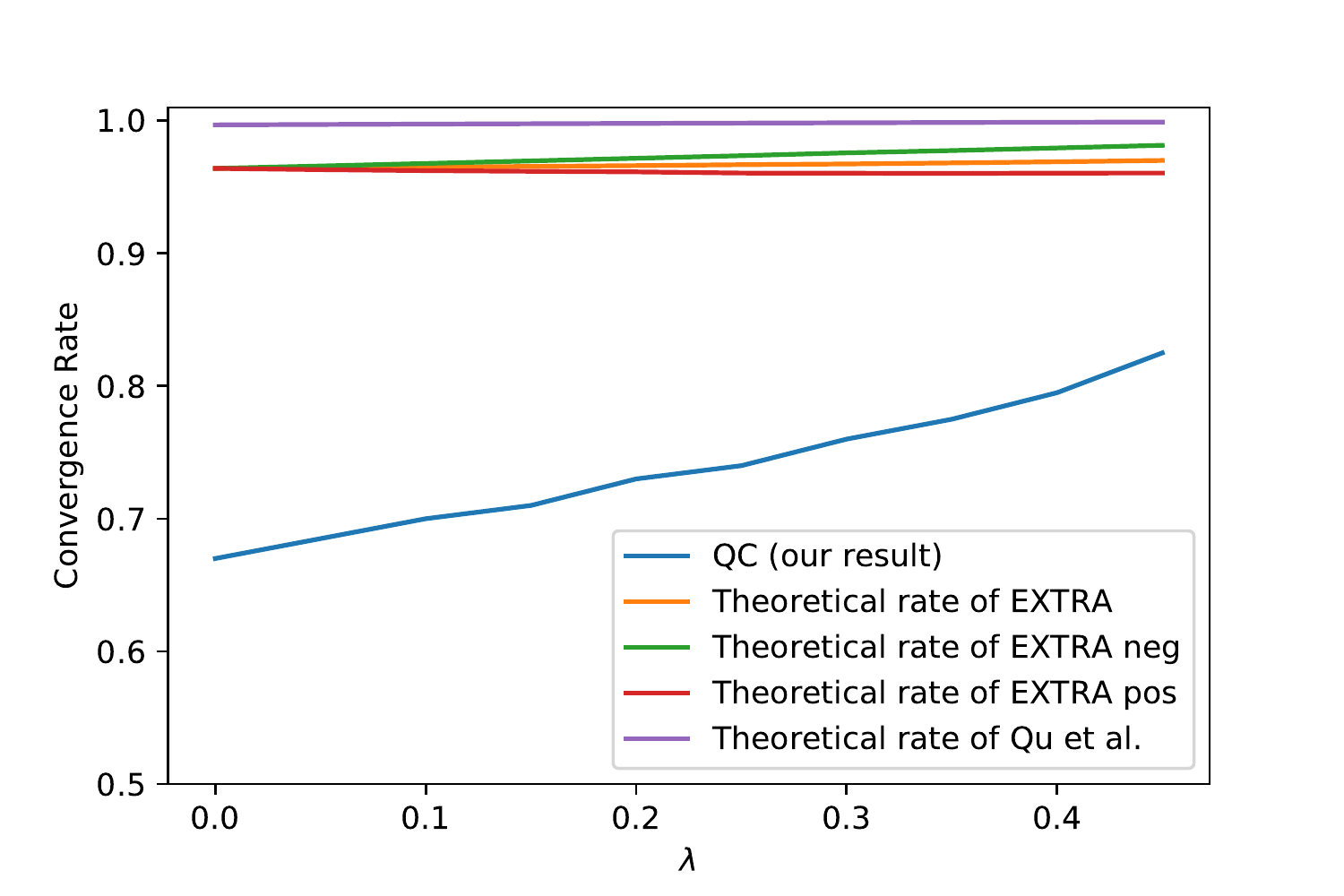}
    \caption{Comparison of the convergence rates for different methods.}
\label{comparison}
\end{subfigure}
\hfill
  \caption{Optimal convergence rate for Distributed MD obtained by solving LMIs under various assumptions.}
  \label{fig:more_exp}
\end{figure*}

We finally remark that common analysis on distributed MD involves general primal-dual norms \cite{shahrampour2017distributed}, whereas QCs are defined with respect to the Euclidean norm. The use of general primal-dual norms in non-strongly convex problems helps with improving the rate up to a multiplicative factor of $\sqrt{d}$. However, in strongly convex case the rate is exponentially fast, and a more general analysis can only change the iteration complexity by at most logarithmic factors of $d$, which is an interesting avenue to investigate in the future. 

{\bl
\subsection{$O(1/k)$ Convergence for Convex Functions}
In the following theorem, we present the counterpart of Theorem \ref{distributed_theorem} for convex problems.
\begin{theorem}\label{distributed_theorem_convex}
Let Assumptions \ref{Assumption_phi} and \ref{network_assump} hold and assume all local functions $f_i$ are convex ($\mu_f=0$) and $L_f$-smooth. Recall the definitions of matrices $A_1, A_2, B_1, B_2, H_1, H_2, M_f, M_\lambda, M_\phi$ in Theorem \ref{distributed_theorem} and define the following additional matrices,
\begin{align*}
    M_1 &=  \begin{bmatrix}
    0&0&0&0&0\\
    0&0&0&0&0\\
    0&0&{L_f}&0&0\\
    0&0&0&0&0\\
    0&0&0&0&0
    \end{bmatrix} ~~~ M_2 &=  \begin{bmatrix}
    0&0&0&0&0\\
    0&0&0&0&0\\
    0&0&0&\frac{1}{2}&0\\
    0&0&\frac{1}{2}&0&0\\
    0&0&0&0&0
    \end{bmatrix}.
\end{align*}
If there exist $\eta_1 \geq 0, P \in \mathbb{S}^{2}, P \succ 0, {\bl \Sigma_{eq}\in \mathbb{S}^{2},} \sigma_f \geq 0, \sigma_{\phi} \geq 0, \sigma_{\lambda} \geq 0, \epsilon\geq 0$, such that the following matrix inequalities hold for $i = 1, 2$

\begin{equation}
    \begin{aligned}
        \label{distributed_LMI_convex}
    &\begin{bmatrix}
     A_i^\top P A_i -  P & A_i^\top P B_i \\ B_i^\top  P A_i & B_i^\top  P B_i
    \end{bmatrix} \!+\! \sigma_{f} M_{f} + \sigma_{\lambda} M_{\lambda}  \\
    &+ \sigma_{\phi}  M_{\phi}  + \epsilon  M_i     + H_i^\top \Sigma_{eq} H_i
    \preceq 0,
    \end{aligned}
\end{equation}
then, the iterates of the distributed MD algorithm \eqref{discrete_distributed} initialized at $y^{(0)}=0$ satisfy the following inequality,
$$\sum\limits_{i=1}^n \Big( f(\bar{x}_i^{(K)}) - f^\star\Big) \leq \frac{V^{(0)}}{\epsilon  K},$$
where $\bar{x}_i^{(K)}\triangleq\frac{ 1}{K} \sum\limits_{k=0}^{K-1} x_i^{(k)}$.
\end{theorem}
We refer to the appendix of \cite{sun2021centralized} for the proof of this theorem. Given that $ f(\bar{x}_i^{(K)}) - f^\star$ is non-negative, it is easy to see that the function evaluated at the ergodic average of each agent iterate converges to minimum with a rate of $O(1/K).$
}

\subsection{Evaluating the Tightness of Results}

For the distributed MD algorithm, we provide numerical results based on Theorem \ref{distributed_theorem}. First, we demonstrate the influence of the network structure, and then we compare the rate recovered by Theorem \ref{distributed_theorem} to existing theoretical rates on distributed GD when it achieves exponential convergence.   

\subsubsection{Impact of the Network Structure on Convergence Rate}

We calculate the worst-case convergence rate with several choices of $\lambda$ and plot it with respect to the step-size $\eta_1$. We set the local functions to have condition number $\kappa_f = 2$ and the DGF to have condition number $\kappa_\phi = 2$. Each curve in the plot represents a certain $\lambda$ and is obtained by scanning feasible values for the decision variables in the LMIs \eqref{distributed_LMI}. From Fig. \ref{lambda_dif}, we can see that there exists an optimal step-size to obtain the best convergence rate, and that as $\lambda$ increases, the best rate becomes worse. Hence, for any given network structure and its corresponding Laplacian matrix, we should select $\eta_2$ such that $\lambda$ is minimized. This is consistent with results on distributed optimization, where having a {\bl larger $\lambda$ deteriorates the performance.} 

{\bl In Fig. \ref{fig:1b}, we keep $\kappa_\phi = 5$ constant and study the optimal convergence rate for different $\lambda$ and $\kappa_f$. When the condition number increases, the optimal rate worsens. This behavior aligns with gradient descent, where $\kappa_\phi = 1$.}

\subsubsection{Comparison with Distributed Gradient Descent}
To the best of our knowledge, there is currently no work that provides an exponential convergence rate for distributed MD algorithm. Hence, we select two previous works on distributed GD, namely \cite{shi2015extra} and \cite{qu2017harnessing}, and compare our performance with the theoretical rates provided in these works. In order to provide a fair comparison, we must set $\kappa_\phi = 1$ to ensure that MD reduces to GD. We also set the local functions to have condition number $\kappa_f = 3$. 

Of the two related works above, EXTRA \cite{shi2015extra} is of particular relevance to our algorithm. If the matrix $\Tilde{W}$ in EXTRA is set to be $\frac{I_n + W}{2}$, the EXTRA algorithm coincides with our algorithm with the exception of having a coefficient difference of $\frac{1}{2}$ for the tracking term. Note that the theoretical convergence rate of EXTRA relies on the spectral norm of $\Delta W$ as well as the {\it smallest non-zero eigenvalue} $\lambda_{n}$ of $W$. We plot the convergence rate of EXTRA under three different scenarios: \begin{enumerate}
    \item $\lambda_n = \lambda$, (EXTRA pos)
    \item $\lambda_n = - \lambda$, (EXTRA neg)
    \item $\lambda_n \approx 0$, (EXTRA)
\end{enumerate}

From Fig. \ref{comparison}, we can see that when $\lambda$ is small, the rate recovered by Theorem \ref{distributed_theorem} significantly outperforms EXTRA. As $\lambda$ increases, the convergence rate calculated for our method starts increasing. We also include the theoretical convergence results from Qu et al. \cite{qu2017harnessing}, which is consistently outperformed by EXTRA.

Note that the point of this plot is not to declare a winner among algorithms. The goal is to show that the richness of the Lyapunov function and QC analysis provides a machinery to obtain better convergence rates, especially compared to the rates that are algorithm specific. In this case, our algorithm can coincide with EXTRA, but still our analysis provides better rates. Our observation is in line with empirical results of \cite{sundararajan2017robust}.

\section{Conclusion}
In this paper, we proposed a SDP framework to characterize the convergence rate of the mirror descent algorithm for both centralized and distributed settings{\bl , and empirical evaluations were performed} under the assumption of strongly convex and smooth local objective functions. For the centralized case, we derived a closed-form feasible solution to the SDP for the convergence rate, which depends on the condition number of the distance generating function. For the decentralized case, we numerically derived the convergence rates using SDP. These SDPs do not scale with the ambient dimension and the network size. Using the QC framework, we further proved the $O(1/k)$ convergence rate for centralized and distributed MD in the convex and smooth setting. It would be interesting to derive analytical rates for the distributed case. Another important direction is the analysis of the mirror descent algorithm with primal-dual norms. This is a challenging problem as current SDP approaches rely on the Euclidean norm and they do not lend themselves to general primal-dual norms.

\bibliographystyle{IEEEtran}
\bibliography{ref}




\section*{Appendix}
\subsection{\bl Preliminary Lemmas for Proof of Theorems}
{\bl
In this section we provide a few lemmas used in the proof of main theorems later.
}
\begin{lemma}\label{bregman_difference}
Let Assumption \ref{Assumption_phi} hold and consider the Lyapunov function $V^{(k)} = \rho^{-k} \Dc_{\phi^\star}(z^{(k)} , z^\star)$. Then, the following inequality,
$$V^{(k+1)} - V^{(k)} \leq \rho^{-k-1} e^{(k)\top}{\bl M_{sc}} e^{(k)},$$
is satisfied, where ${\bl M_{sc}}$ is given in Theorem \ref{centralized_theorem}, and $e^{(k)}$ is defined in \eqref{eq:stack}.
\end{lemma}
\begin{proof}
From the definition of Lyapunov function and Bregman divergence, we have that
\begin{align*}
    &V^{(k+1)} - V^{(k)}\\
    =&\rho^{-k-1}\mathcal{D}_{\phi^\star}(z^{(k+1)}, z^\star) - \rho^{-k}\mathcal{D}_{\phi^\star}(z^{(k)}, z^\star) \\
    =& \rho^{-k-1}(\phi^\star(z^{(k+1)}) - \phi^\star(z^\star) - \langle \nabla\phi^\star(z^\star), z^{(k+1)}-z^\star \rangle) - \\
    &\rho^{-k}(\phi^\star(z^{(k)}) - \phi^\star(z^\star) - \langle \nabla\phi^\star(z^\star), z^{(k)}-z^\star \rangle)\\
    =& \rho^{-k-1}\phi^\star(z^{(k+1)}) -\rho^{-k-1}\langle x^\star, z^{(k)} - z^\star -\eta u^{(k)} \rangle\\
    +&\rho^{-k}\langle x^\star, z^{(k)} - z^\star \rangle - (\rho^{-k-1}-\rho^{-k})\phi^\star(z^\star)-\rho^{-k}\phi^\star(z^{(k)}).
    \end{align*}
    Since $\phi^\star$ is $\mu_\phi^{-1}$-smooth, we get 
    \begin{align*}
    &V^{(k+1)} - V^{(k)}\\
    &\leq  \rho^{-k-1}[\phi^\star(z^{(k)}) + \langle x^{(k)}, -\eta u^{(k)} \rangle + \frac{\eta^2}{2\mu_\phi}\|u^{(k)}\|^2] \\
   \vphantom{\frac{\eta^2}{2\mu_\phi}} &-\rho^{-k-1}\langle x^\star, z^{(k)} - z^\star -\eta u^{(k)} \rangle+\rho^{-k}\langle x^\star, z^{(k)} - z^\star \rangle \\
   \vphantom{\frac{\eta^2}{2\mu_\phi}} &- (\rho^{-k-1}-\rho^{-k})\phi^\star(z^\star)-\rho^{-k}\phi^\star(z^{(k)})\\
    \vphantom{\frac{\eta^2}{2\mu_\phi}}&= (\rho^{-k-1}-\rho^{-k})(\phi^\star(z^{(k)})-\phi^\star(z^\star)) + \frac{\rho^{-k-1}\eta^2}{2\mu_\phi}\|u^{(k)}\|^2 \\
   \vphantom{\frac{\eta^2}{2\mu_\phi}} &-(\rho^{-k-1}-\rho^{-k})\langle x^\star, z^{(k)} - z^\star  \rangle- \rho^{-k-1}\langle x^{(k)}-x^\star, \eta u^{(k)} \rangle.
   \end{align*}
   Applying smoothness again, we can bound $V^{(k+1)} - V^{(k)}$ by
   {\small 
   \begin{align*}
    & (\rho^{-k-1}-\rho^{-k})(\nabla\phi^\star(z^\star)^\top (z^{(k)} - z^\star) + \frac{1}{2\mu_\phi}\|z^{(k)} - z^\star\|^2) \\
  \vphantom{\frac{\eta^2}{2\mu_\phi}}  &+ \rho^{-k-1}\langle x^{(k)}-x^\star, -\eta u^{(k)} \rangle + \frac{\rho^{-k-1}\eta^2}{2\mu_\phi}\|u^{(k)}\|^2\\
   \vphantom{\frac{\eta^2}{2\mu_\phi}} &-(\rho^{-k-1}-\rho^{-k})\langle x^\star, z^{(k)} - z^\star  \rangle\\
    \vphantom{\frac{\eta^2}{2\mu_\phi}}&= \rho^{-k-1}( \frac{1-\rho}{2\mu_\phi}\|z^{(k)} - z^\star\|^2 -\eta \langle x^{(k)}-x^\star,  u^{(k)} \rangle + \frac{\eta^2}{2\mu_\phi}\|u^{(k)}\|^2)\\
    \vphantom{\frac{\eta^2}{2\mu_\phi}}&= \rho^{-k-1} e^{(k)\top}{\bl M_{sc}} e^{(k)},
\end{align*}}
and observing $u^\star=0$ finishes the proof.
\end{proof}
{\bl

\begin{lemma}\label{bregman_difference_convex}
Let Assumption \ref{Assumption_phi} hold and consider the Lyapunov function $V^{(k)} = \epsilon \sum_{i=0}^{k-1} (f(x^{(i)})-f(x^\star)) + \Dc_{\phi^\star}(z^{(k)},z^\star)$, defined for $\epsilon>0$. Then, when $f$ is convex, the following inequality holds
$$V^{(k+1)} - V^{(k)} \leq  e^{(k)\top} M_{c} e^{(k)},$$
where $M_c$ is given in Theorem \ref{centralized_theorem_convex}, and $e^{(k)}$ is defined in \eqref{eq:stack}.
\end{lemma}
\begin{proof}
Following the proof of Lemma \ref{bregman_difference} and by setting $\rho = 1$, we know that
\begin{equation*}\resizebox{0.99\hsize}{!}{$
\Dc_{\phi^\star}(z^{(k+1)},z^\star) -  \Dc_{\phi^\star}(z^{(k)},z^\star) \leq   -\eta \langle x^{(k)}-x^\star,  u^{(k)} \rangle + \frac{\eta^2}{2\mu_\phi}\|u^{(k)}\|^2.
$}
\end{equation*}
Therefore, we can bound $V^{(k+1)} - V^{(k)}$ using the convexity of $f$ and observing $u^\star=0$, as follows
   {\small 
   \begin{align*}
    & -\eta \langle x^{(k)}-x^\star,  u^{(k)} \rangle + \frac{\eta^2}{2\mu_\phi}\|u^{(k)}\|^2 +  \epsilon (f(x^{(k)})-f(x^\star))\\
    \leq & -\eta \langle x^{(k)}-x^\star,  u^{(k)} \rangle + \frac{\eta^2}{2\mu_\phi}\|u^{(k)}\|^2 + \epsilon \langle u^{(k)} - u^\star, x^{(k)} - x^\star\rangle \\
    =& e^{(k)\top}{ M_{c}} e^{(k)}.
\end{align*}}
\end{proof}

\begin{lemma}\label{lemma_distributed_bound_f}
Assume all local functions $f_i$ are convex ($\mu_f=0$) and $L_f$-smooth. Then, the following inequality holds for the distributed algorithm in \eqref{DMD_spectral}
\begin{align*}
     \sum_{i=1}^n(f({x_i}^{(k)}) - f^\star )  \leq e^{(k)\top} M e^{(k)},
\end{align*}
where $f^\star\triangleq \min_x f(x)$ and $M \in \mathbb{R}^{5nd\times 5nd}$ is defined as
\begin{align*}
    M\triangleq\begin{bmatrix}
    0&0&0&0&0\\
    0&0&0&0&0\\
    0&0&L_f( {I}_n-\frac{1}{n}\mathbf{1}_n\mathbf{1}_n^\top )\otimes  I_d &\frac{1}{2n}\mathbf{1}_n\mathbf{1}_n^\top\otimes  I_d &0\\
    0&0&\frac{1}{2n}\mathbf{1}_n\mathbf{1}_n^\top\otimes  I_d&0&0\\
    0&0&0&0&0
    \end{bmatrix}.
\end{align*}
\end{lemma}
\begin{proof}
Recall that we denote an optimal solution of the function in \eqref{mainproblem} as $x_\star \in \R^d$. From the definition of $f$, we know that $\sum_{i=1}^n \nabla f_i(x_\star) = 0.$ We note that the dimension of $x_\star$ differs from that of the stationary point of the distributed system $x^\star \in \R^{nd}$. Specifically, we have $x^\star \triangleq \mathbf{1}_n  \otimes x_\star.$

For any $x_j^{(k)}$ at agent $j$, we have that
\begin{equation*}\resizebox{0.99\hsize}{!}{$
    \begin{aligned}
    &n(f({x_j}^{(k)}) - f^\star)=\sum_{i=1}^n (f_i({x_j}^{(k)}) - f_i(x_\star))\\
    = &\sum_{i=1}^n \Big(f_i({x_j}^{(k)}) - f_i(x_{i}^{(k)})+f_i(x_{i}^{(k)}) - f_i(x_\star)\Big)\\
    \leq &\sum_{i=1}^n \Big(\nabla f_i(x_{i}^{(k)})^\top ({x_j}^{(k)} - x_{i}^{(k)}) + \frac{L_f}{2}\|{x_j}^{(k)} - x_{i}^{(k)}\|^2+f_i(x_{i}^{(k)}) - f_i(x_\star)\Big)\\
    \leq & \sum_{i=1}^n \Big(\nabla f_i(x_{i}^{(k)})^\top ({x_j}^{(k)} - x_{i}^{(k)}) + \frac{L_f}{2}\|{x_j}^{(k)} - x_{i}^{(k)}\|^2+\nabla f_i(x_{i}^{(k)})^\top (x_{i}^{(k)} - x_\star)\Big)\\
    =&\sum_{i=1}^n \Big(\nabla f_i(x_{i}^{(k)})^\top ({x_j}^{(k)} - x_{i}^{(k)}+x_{i}^{(k)} - x_\star) + \frac{L_f}{2}\|{x_j}^{(k)} - x_{i}^{(k)}\|^2\Big)\\
    =&\sum_{i=1}^n \Big(\big(\nabla f_i(x_{i}^{(k)}) - \nabla f_i (x_\star)\big)^\top ({x_j}^{(k)} - x_\star) + \frac{L_f}{2}\|{x_j}^{(k)} - x_{i}^{(k)}\|^2\Big),
    \end{aligned}$}
\end{equation*}
where the two inequalities are induced by the Lipschitz-smoothness and convexity of $f_i$, respectively. Since $x_\star$ is a global optimal solution, we also have
$$\sum_{i=1}^n (  \nabla f_i (x_\star)^\top ({x_j}^{(k)} - x_\star)) = (\sum_{i=1}^n   \nabla f_i (x_\star))^\top ({x_j}^{(k)} - x_\star) = 0 .$$
Summing over $j$, we get
\begin{equation*}
    \begin{aligned}
         &n\sum_{j=1}^{n}(f({x_j}^{(k)}) - f^\star) \\
         \leq & \sum_{j=1}^{n} \sum_{i=1}^n \big(\nabla f_i(x_{i}^{(k)}) - \nabla f_i (x_\star)\big)^\top({x_j}^{(k)} - x_\star) \\
         & + \sum_{j=1}^{n} \sum_{i=1}^n\frac{L_f}{2}\|{x_j}^{(k)} - x_{i}^{(k)}\|^2.
    \end{aligned}
\end{equation*}
Writing above in matrix form and dividing by $n$, we derive
\begin{equation*}
    \begin{aligned}
         &\sum_{j=1}^{n}(f({x_j}^{(k)}) - f^\star)  \\
         =&(u^{(k)} - u^\star)^\top ( \frac{1}{n} \mathbf{1}_n\mathbf{1}_n^\top\otimes  I_d )(x^{(k)} - x^\star)\\
         &+ {L_f}(x^{(k)} - x^\star)^\top \Big(( {I}_n-\frac{1}{n}\mathbf{1}_n\mathbf{1}_n^\top )\otimes  I_d \Big)(x^{(k)} - x^\star)\\
        =& e^{(k)\top} M e^{(k)}.
    \end{aligned}
\end{equation*}
\end{proof}
}

\subsection{Proof of Proposition \ref{SDP_transform}}
We start with the following lemma, which helps with turning the non-affine constraint to an affine constraint in the SDP.

\begin{lemma}\label{schur}
If matrix $M \in \R^{n\times n}$ can be decomposed as $M = N + S S^\top$, where $S \in \R^{n\times m}$, then a negative semi-definite constraint on $M$ can be equivalently represented by an affine constraint on $N$ and $S$.
\end{lemma}
\begin{proof}
Consider the following matrix $M' \in \R^{(n+m)\times (n+m)}$
$$M'=\begin{bmatrix}
        -N&S\\S^\top&I_m
    \end{bmatrix}.$$
    By properties of Schur complement, we have that
    $$M' \succeq 0 \iff -N - SS^\top \succeq 0 \iff M \preceq 0.$$
    Therefore, we can equivalently use $M' \succeq 0$ as the constraint (in lieu of $M \preceq 0$). This constraint is affine with respect to both $N$ and $S$.
\end{proof}
We now provide the proof for Proposition \ref{SDP_transform}.
\begin{proof}
For brevity, in this proof we use $I=I_d$. Given the matrices defined in Theorem \ref{centralized_theorem}, we can write the last LMI in \eqref{centralized_optimization} as 
\begin{align*}
        \begin{bmatrix}\frac{1-\rho}{2\mu_\phi}I&0&0\\
     0&0&\frac{-\eta}{2} I\\
     0&\frac{-\eta}{2} I&\frac{\eta^2}{2\mu_{\phi}}I
     \end{bmatrix}   
      + \sigma_f\begin{bmatrix}
            0&0&0\\
            0&\frac{-\mu_f L_f}{\mu_f + L_f}I&\frac{I}{2}\\
            0&\frac{I}{2}&\frac{-1}{\mu_f + L_f}I
        \end{bmatrix} &\\
        + \sigma_\phi \begin{bmatrix}
    \frac{-1}{\mu_\phi + L_\phi}I&\frac{I}{2}&0\\
    \frac{I}{2}&\frac{-\mu_\phi L_\phi}{\mu_\phi + L_\phi}I&0\\
    0&0&0
    \end{bmatrix}   &\preceq 0,
    \end{align*}
    which implies
    \begin{align*}
    \begin{bmatrix}\frac{1-\rho}{2\mu_\phi}I&0&0\\
     0&\frac{-\mu_\phi}{2}I&0\\
     0&0&0
     \end{bmatrix} + \sigma_f\begin{bmatrix}
            0&0&0\\
            0&\frac{-\mu_f L_f}{\mu_f + L_f}I&\frac{I}{2}\\
            0&\frac{I}{2}&\frac{-1}{\mu_f + L_f}I
        \end{bmatrix}    &\\
        + \sigma_\phi \begin{bmatrix}
    \frac{-1}{\mu_\phi + L_\phi}I&\frac{I}{2}&0\\
    \frac{I}{2}&\frac{-\mu_\phi L_\phi}{\mu_\phi + L_\phi}I&0\\
    0&0&0
    \end{bmatrix}  + \begin{bmatrix}
     0\\ \frac{-\sqrt{\mu_\phi}I}{\sqrt{2}}\\ \frac{\eta I }{\sqrt{2 \mu_\phi}}
     \end{bmatrix} \begin{bmatrix}
     0\\ \frac{-\sqrt{\mu_\phi}I}{\sqrt{2}}\\ \frac{\eta I }{\sqrt{2 \mu_\phi}}
     \end{bmatrix}^\top &\preceq 0
\end{align*}
We can then remove $I$ inside the block matrix elements in the equation above and apply Lemma \ref{schur} to get
\begin{align*}
      \begin{bmatrix}\frac{\rho-1}{2\mu_\phi}&0&0\\
     0&\frac{\mu_\phi }{2}&0\\
     0&0&0
     \end{bmatrix}  
     +\begin{bmatrix}
            0&0&0\\
            0& \frac{\mu_f L_f\sigma_f}{\mu_f + L_f}&\frac{- \sigma_f }{2}\\
            0&\frac{- \sigma_f }{2}& \frac{\sigma_f}{\mu_f + L_f}
        \end{bmatrix} &\\
         +\begin{bmatrix}
    \frac{\sigma_\phi}{\mu_\phi + L_\phi}&\frac{- \sigma_\phi }{2}&0\\
    \frac{- \sigma_\phi}{2}& \frac{\mu_\phi L_\phi\sigma_\phi}{\mu_\phi + L_\phi}&0\\
    0&0&0
    \end{bmatrix}  
     - \begin{bmatrix}
     0\\ \frac{-\sqrt{\mu_\phi}}{\sqrt{2}}\\ \frac{\eta  }{\sqrt{2 \mu_\phi}}
     \end{bmatrix} \begin{bmatrix}
     0\\ \frac{-\sqrt{\mu_\phi}}{\sqrt{2}}\\ \frac{\eta  }{\sqrt{2 \mu_\phi}}
     \end{bmatrix}^\top &\succeq 0\\
     \Longrightarrow 
      \resizebox{0.82\hsize}{!}{$\begin{bmatrix}
     \frac{\sigma_\phi}{\mu_\phi + L_\phi} + \frac{\rho-1}{2\mu_\phi}&\frac{- \sigma_\phi }{2}&0\\
     \frac{- \sigma_\phi }{2}&\frac{\mu_\phi L_\phi\sigma_\phi}{\mu_\phi + L_\phi} +\frac{\mu_\phi }{2} + \frac{\mu_f L_f\sigma_f}{\mu_f + L_f}&\frac{- \sigma_f }{2}\\
     0&\frac{- \sigma_f }{2}&\frac{\sigma_f}{\mu_f + L_f}
     \end{bmatrix}$} &\\
     - \begin{bmatrix}
     0\\ \frac{-\sqrt{\mu_\phi}}{\sqrt{2}}\\ \frac{\eta  }{\sqrt{2 \mu_\phi}}
     \end{bmatrix} \begin{bmatrix}
     0\\ \frac{-\sqrt{\mu_\phi}}{\sqrt{2}}\\ \frac{\eta  }{\sqrt{2 \mu_\phi}}
     \end{bmatrix}^\top &\succeq 0\\
     \Longrightarrow 
      \resizebox{0.82\hsize}{!}{$\begin{bmatrix}
     \frac{\sigma_\phi}{\mu_\phi + L_\phi} + \frac{\rho-1}{2\mu_\phi}&\frac{- \sigma_\phi }{2}&0&0\\
     \frac{- \sigma_\phi }{2}&\frac{\mu_\phi L_\phi\sigma_\phi}{\mu_\phi + L_\phi} +\frac{\mu_\phi }{2} + \frac{\mu_f L_f\sigma_f}{\mu_f + L_f}&\frac{- \sigma_f }{2}&\frac{-\sqrt{\mu_\phi}}{\sqrt{2}}\\
     0&\frac{- \sigma_f }{2}&\frac{\sigma_f}{\mu_f + L_f}&\frac{\eta  }{\sqrt{2 \mu_\phi}}\\
     0&\frac{-\sqrt{\mu_\phi}}{\sqrt{2}}&\frac{\eta  }{\sqrt{2 \mu_\phi}}&1
     \end{bmatrix}$} &\succeq 0,
\end{align*}
thereby completing the proof. 
\end{proof}
\subsection{Proof of Proposition \ref{optrho}}
If $ \eta = \sigma_f = \frac{2 \mu_\phi}{\mu_f + L_f}$, the LMI in \eqref{centralized_optimization} becomes

\begin{align*}
        \begin{bmatrix}\frac{(1-\rho)}{2\mu_\phi}I_d&0&0\\0&0&\frac{- \mu_\phi}{\mu_f + L_f} I_d\\0&\frac{ -\mu_\phi}{\mu_f + L_f} I_d&\frac{2 \mu_\phi}{(\mu_f + L_f)^2}I_d\end{bmatrix}&\\
        + \sigma_{\phi} 
        \begin{bmatrix}   \frac{-1}{\mu_\phi + L_\phi}I_d&\frac{1}{2}I_d&0\\    \frac{1}{2}I_d&\frac{-\mu_\phi L_\phi}{\mu_\phi + L_\phi}I_d&0\\    0&0&0    \end{bmatrix} &\\
        + \begin{bmatrix}    0&0&0\\    0&\frac{2 \mu_\phi}{\mu_f + L_f} \frac{-\mu_f L_f}{\mu_f + L_f}I_d&\frac{ \mu_\phi}{\mu_f + L_f}I_d\\0&\frac{ \mu_\phi}{\mu_f + L_f}I_d&\frac{2 \mu_\phi}{\mu_f + L_f}\frac{-1}{\mu_f + L_f}I_d    \end{bmatrix}  
        &\preceq 0,  \\
        \end{align*}
        which, after removing $I_d$, simplifies to 
        \begin{equation*}
        \resizebox{0.99\hsize}{!}{$ 
        \begin{bmatrix}\frac{(1-\rho)}{2\mu_\phi}&0&0\\0& \frac{-2 \mu_\phi \mu_f L_f}{(\mu_f + L_f)^2}&0\\0&0&0\end{bmatrix} 
        + \sigma_{\phi} 
        \begin{bmatrix}   \frac{-1}{\mu_\phi + L_\phi}&\frac{1}{2}&0\\    \frac{1}{2}&\frac{-\mu_\phi L_\phi}{\mu_\phi + L_\phi}&0\\    0&0&0    \end{bmatrix} 
        \preceq 0, $}
        \end{equation*}
        and we get 
        \begin{align*}
        \begin{bmatrix}\frac{(1-\rho)}{2\mu_\phi}&0\\0& \frac{-2 \mu_\phi \mu_f L_f}{(\mu_f + L_f)^2}\end{bmatrix} 
        + \sigma_{\phi} 
        \begin{bmatrix}   \frac{-1}{\mu_\phi + L_\phi}&\frac{1}{2}\\    \frac{1}{2}&\frac{-\mu_\phi L_\phi}{\mu_\phi + L_\phi} \end{bmatrix} 
        &\preceq 0  \\
        \iff
        \begin{bmatrix}\frac{(1-\rho)}{2\mu_\phi}-\sigma_{\phi}  \frac{1}{\mu_\phi + L_\phi}& \frac{\sigma_{\phi}}{2}\\ \frac{\sigma_{\phi}}{2}& \frac{-2 \mu_\phi \mu_f L_f}{(\mu_f + L_f)^2}-\sigma_{\phi} \frac{\mu_\phi L_\phi}{\mu_\phi + L_\phi} \end{bmatrix} 
        &\preceq 0
    \end{align*}
This is equivalent to the following constraints on the principal minors of the matrix:
\begin{enumerate}
    \item \begin{equation*}
        \begin{aligned}
            &-\frac{(1-\rho)}{2\mu_\phi}+\sigma_{\phi}  \frac{1}{\mu_\phi + L_\phi} \geq 0
        \end{aligned}
    \end{equation*}
    \item \begin{equation*}
        \begin{aligned}
            &\frac{2 \mu_\phi \mu_f L_f}{(\mu_f + L_f)^2} + \sigma_{\phi} \frac{\mu_\phi L_\phi}{\mu_\phi + L_\phi} \geq 0
        \end{aligned}
    \end{equation*}
    \item 
    \begin{equation*}
        \begin{aligned}
            &\bigg(-\frac{(1-\rho)}{2\mu_\phi}+\sigma_{\phi}  \frac{1}{\mu_\phi + L_\phi}\bigg)\bigg(\frac{2 \mu_\phi \mu_f L_f}{(\mu_f + L_f)^2} \\
            & +\sigma_{\phi} \frac{\mu_\phi L_\phi}{\mu_\phi + L_\phi}\bigg)-  \frac{\sigma_{\phi}^2}{4}\geq 0
        \end{aligned}
    \end{equation*}
\end{enumerate}
The last constraint is the most strict of all constraints. Hence, we will focus on the last constraint, where we can alternatively write
$$\rho \geq 1-\frac{2\sigma_{\phi}}{1+\kappa_\phi} + \frac{\sigma_{\phi}^2}{2}\left(\frac{2  \mu_f L_f}{(\mu_f + L_f)^2} + \sigma_{\phi} \frac{\kappa_\phi}{1+\kappa_\phi}\right)^{-1}.$$
The right-hand side can be seen as a function of $\sigma_{\phi}$; it takes its minimum when derivative of $\sigma_{\phi}$ is zero. We denote the optimal $\sigma_{\phi}$ by $\sigma_{\phi}^\star$. Therefore,
$$ \frac{d}{d\sigma_{\phi}}\bigg(1-\frac{2\sigma_{\phi}}{1+\kappa_\phi} + \frac{\sigma_{\phi}^2}{2}\bigg(\frac{2  \mu_f L_f}{(\mu_f + L_f)^2} + \sigma_{\phi} \frac{\kappa_\phi}{1+\kappa_\phi}\bigg)^{-1}\bigg) =0$$
The positive solution for the equation above is
$$\sigma_{\phi}^\star = \frac{4 \mu_f L_f}{(\mu_f + L_f)^2} \frac{(1 + \kappa_\phi)}{ \kappa_\phi (\kappa_\phi -1)},   $$
and the corresponding solution for $\rho$ is
\begin{align*}
    \rho_{opt} &= 1-\frac{2\sigma_{\phi}^\star}{1+\kappa_\phi} + \frac{\sigma_{\phi}^{\star2}}{2}\bigg(\frac{2  \mu_f L_f}{(\mu_f + L_f)^2} + \sigma_{\phi}^\star \frac{\kappa_\phi}{1+\kappa_\phi}\bigg)^{-1}\\
    &=1- \frac{4 \mu_f L_f}{(\mu_f + L_f)^2 \kappa_\phi^2},
\end{align*}
thereby completing the proof. \qed
{\bl
\subsection{Proof of Theorem \ref{centralized_theorem_convex}}
\begin{proof}
We consider the following Lyapunov candidate
$$V^{(k)} = \epsilon \sum_{i=0}^{k-1} (f(x^{(i)})-f(x^\star)) + \Dc_{\phi^\star}(z^{(k)},z^\star).$$
Using Lemma \ref{bregman_difference_convex}, we can calculate an upper bound for the following term
\begin{equation}
    \begin{aligned}
        V^{(k+1)} - V^{(k)} \leq  e^{(k)\top} M_{c} e^{(k)}.
    \end{aligned}
\end{equation}
Combined with the two QCs, the above implies that
\begin{equation}
    \begin{aligned}
        &\hphantom{=\ }V^{(k+1)} - V^{(k)} \leq  e^{(k)\top}M_c e^{(k)}\\
        &\leq  e^{(k)\top}M_c e^{(k)} + \sigma_f e^{(k)\top}M_f e^{(k)} + \sigma_\phi e^{(k)\top}M_\phi e^{(k)}\\
        &= e^{(k)\top}(M_c  + \sigma_f M_f  + \sigma_\phi M_\phi )e^{(k)}.
    \end{aligned}
\end{equation}
If the LMI in \eqref{centralized_LMI_convex} is feasible, then the Lyapunov function satisfies $V^{(k+1)} \leq V^{(k)}$, which is equivalent to
\begin{equation}
     \Dc_{\phi^\star}(z^{(k+1)},z^\star) -  \Dc_{\phi^\star}(z^{(k)},z^\star)  \leq - \epsilon (f(x^{(k)}) - f(x^\star)).
\end{equation}
Summing up both sides and rearranging terms, we obtain
$$\frac{\sum_{i = 1}^K (f(x^{(i)}) - f(x^\star))}{K} \leq \frac{ \Dc_{\phi^\star}(z^{(0)},z^\star)}{\epsilon K}.$$
The left hand side is again lower bounded by $ f(\Bar{x}^{(K)}) - f(x^\star)$ due to the convexity of $f$, which completes the proof.
\end{proof}

\subsection{Proof of Theorem \ref{distributed_theorem_convex}}
\begin{proof}

Recalling Proposition \ref{IQC}, based on the assumptions, we have that
\begin{align*}
     e^{(k)\top} (M_f\otimes I_{nd}) e^{(k)} &\geq 0,\\
     e^{(k)\top} (M_\phi\otimes I_{nd}) e^{(k)} &\geq 0.
\end{align*}
Note that for the mapping $z \mapsto \Delta W z$, given that $\bl \lambda = \norm{\Delta W}$, we can write
\begin{align*}
   e^{(k)\top} (M_\lambda\otimes I_{nd}) e^{(k)} &\geq 0.
\end{align*}
Using Lemma \ref{lemma_distributed_bound_f}, we know that 
\begin{align*}
     \sum_{i=1}^n(f({x_i}^{(k)}) - f^\star )  \leq e^{(k)\top} M e^{(k)},
\end{align*}
where $M \in \mathbb{R}^{5nd\times 5nd}$ is defined as
\begin{align*}
    M\triangleq\begin{bmatrix}
    0&0&0&0&0\\
    0&0&0&0&0\\
    0&0&{L_f}( I_n-\frac{1}{n}\mathbf{1}_n\mathbf{1}_n^\top )\otimes  I_d&\frac{1}{2n}\mathbf{1}_n\mathbf{1}_n^\top\otimes  I_d &0\\
    0&0&\frac{1}{2n}\mathbf{1}_n\mathbf{1}_n^\top\otimes  I_d&0&0\\
    0&0&0&0&0
    \end{bmatrix}.
\end{align*}
Also, from \eqref{C_D}, we have the following equality for any $\Sigma_{eq} \in \mathbb{S}^2$,
\begin{align*}
   e^{(k)\top} H^\top (\Sigma_{eq}\otimes I_{nd}) H e^{(k)}& = 0.
\end{align*}
Now, let us define the Lyapunov function
\begin{align*}
    V^{(k)} = (\xi^{(k)}-\xi^\star)^\top P' (\xi^{(k)}-\xi^\star),
\end{align*}
where $P' = P  \otimes I_{nd}$. Then, using \eqref{new} we can derive
\begin{align*}
    V^{(k+1)} - V^{(k)} = e^{(k)\top}  \begin{bmatrix}
     A^\top P' A - P'  & A^\top P' B \\ B^\top P' A & B^\top P' B
    \end{bmatrix}e^{(k)}.
\end{align*}
If the following LMI holds
{\bl
\begin{equation}
    \begin{aligned}
        \label{distributed_LMI_2_convex}
    &\begin{bmatrix}
     A^\top P' A -   P' & A^\top P' B \\ B^\top  P' A & B^\top  P' B
    \end{bmatrix} + H^\top (\Sigma_{eq}\otimes I_{nd}) H \\
    & +  \epsilon M+(\sigma_{f} M_f+\sigma_{\lambda} M_{\lambda}+ \sigma_{\phi}  M_{\phi})\otimes I_{nd}
    \preceq 0,
    \end{aligned}
\end{equation}
}
then for any $e^{(k)}$, we have that
\begin{equation*}
     e^{(k)\top}  (\begin{bmatrix}
     A^\top P' A - P' & A^\top P' B \\ B^\top P' A & B^\top P' B
    \end{bmatrix}   + \epsilon M)e^{(k)} \leq 0.
\end{equation*}
This inequality implies that
\begin{equation*}
   V^{(k+1)} - V^{(k)} +\epsilon \sum_{i=1}^n(f({x_i}^{(k)}) - f^\star )  \leq 0,
\end{equation*}
due to Lemma \ref{lemma_distributed_bound_f}. By summing up both sides from $k=0$ to $K-1$, applying convexity of $f$ and rearranging, we have
\begin{equation*}
    \sum_{i=1}^n \Big( f(\bar{x}_i^{(K)}) - f^\star \Big) \leq \frac{V^{(0)}}{\epsilon K},
\end{equation*}
where $\bar{x}_i^{(K)}\triangleq\frac{ 1}{K} \sum\limits_{k=0}^{K-1} x_i^{(k)}$. Again, the LMI in \eqref{distributed_LMI_2_convex} can be simplified by defining $J_1, J_2$ in Lemma \ref{lemma_j} similar to the proof of Theorem \ref{distributed_theorem}, which completes the proof.
\end{proof}
}

\end{document}

%% file: header.tex
\newcommand{\0}{\mathbb{0}}

\newcommand{\R}{\mathbb{R}}

\newcommand{\fb}{\mathbf{f}}

\newcommand{\Dc}{\mathcal{D}}
\newcommand{\Ec}{\mathcal{E}}

\newcommand{\Gc}{\mathcal{G}}

\newcommand{\Lc}{\mathcal{L}}
\newcommand{\Nc}{\mathcal{N}}

\newcommand{\Vc}{\mathcal{V}}

\newcommand{\argmin}{\text{argmin}}

\newcommand{\norm}[1]{\left\lVert#1\right\rVert}

\newtheorem{theorem}{Theorem}

\newtheorem{assumption}{Assumption}

\newtheorem{corollary}[theorem]{Corollary}
\newtheorem{definition}{Definition}

\newtheorem{lemma}[theorem]{Lemma}

\newtheorem{proposition}[theorem]{Proposition}
\newtheorem{remark}{Remark}

\newcommand{\bl}{\color{black}}
\newcommand{\rd}{\color{red}}